\documentclass[12pt]{amsart}
\usepackage[top=2.5cm, bottom=2.5cm, left=3cm, right=2.8cm]{geometry}
\usepackage{amsthm}

\usepackage{graphicx}
\usepackage{xcolor}
\usepackage{fancyhdr}
\usepackage{comment}
\usepackage[colorlinks=true, linkcolor=blue!60!black, citecolor=blue!60!black, urlcolor=blue!60!black]{hyperref}

\usepackage{graphicx}
\usepackage{tikz}
\usepackage{xcolor}
\usepackage{fancyhdr}
\theoremstyle{plain}
\newtheorem{maintheorem}{Theorem}

\newtheorem{theorem}{Theorem}[section]
\newtheorem{lemma}[theorem]{Lemma}

\newtheorem{corollary}[theorem]{Corollary}

\theoremstyle{definition}

\newtheorem*{assumption}{Assumption}
\theoremstyle{remark}

\usepackage[colorlinks=true, linkcolor=blue!60!black, citecolor=blue!60!black, urlcolor=blue!60!black]{hyperref}

\begin{document}

\title[An upper bound for the mean Lorentz force of magnetic flows]{An upper bound for the mean Lorentz force of magnetic flows without conjugate points in terms of the geodesic curvature of horocycles}

\author{Anthony J. García}
\address{}
\curraddr{}
\email{}
\thanks{}

\author{Rafael O. Ruggiero}
\address{}
\curraddr{}
\email{}
\thanks{}


\date{}

\dedicatory{}

\commby{}


\begin{abstract}
We study magnetic flows on closed Riemannian surfaces of genus at least two. Let $(M,g)$ be a surface without focal points and let $\Omega$ be a magnetic field. We prove that, above the Mañé critical value, the absence of conjugate points imposes geometric restrictions on the magnetic flow. First, we obtain a bound for the integral of the Lorentz force in terms of the geodesic curvature of the horocycles of the underlying Riemannian metric. As a consequence, a lower bound on the Gaussian curvature yields an explicit bound in terms of the area of the surface. The proofs combine the global geometry of magnetic geodesics with Liouville's formula for the geodesic curvature in the orthogonal coordinates given by the level sets and geodesics of a Busemann function on the universal cover.
\end{abstract}

\maketitle

\section{Introduction}
Magnetic flows model the motion of a charged particle on a Riemannian
manifold under the action of a magnetic field. Given a closed $2$-form
$\Omega$ on $(M,g)$, the flow is the Hamiltonian flow of the kinetic energy
$E(x,v)=\frac12\|v\|_g^2$ with respect to the twisted symplectic form
$\omega_{\mathrm{can}}+\pi^*\Omega$. Unlike the
geodesic flow, its dynamics depends on the energy level, and the relevant dividing line is the Mañé critical value $c(g,\Omega)$. 

Throughout, $(M,g)$ is a closed oriented surface of genus $\geq2$ and
$\Omega$ is a closed $2$-form on $M$. Since $\dim M=2$, there is a unique
function $f\in C^\infty(M)$, the \emph{Lorentz force} of $\Omega$, such that
\[
\Omega=f\,dA_g ,
\]
where $dA_g$ is the Riemannian area form, and a unit-speed magnetic geodesic
$\sigma$ satisfies
\[
k_\sigma(t)=f\bigl(\sigma(t)\bigr).
\]
Magnetic geodesics of energy $\frac12$ are therefore exactly the curves of
prescribed geodesic curvature $f$. This relation suggests that the absence of conjugate points should
impose a restriction on the Lorentz force in terms of the geometry
of $(M,g)$. Therefore, requiring the magnetic flow to have no conjugate points should put a bound on $f$ in terms of the geometry of $(M,g)$.

Our first result makes this precise. We denote by $\mu_l$ the Liouville
measure on $T_1M$, which is preserved by the magnetic flow of $(g,\Omega)$
(see \cite{Paternain2006}). We then have the following theorem:

\begin{maintheorem}\label{Teo A}
Let $(M,g,\Omega)$ be a magnetic system, where $(M,g)$ is a Riemannian
surface without focal points. Suppose that $c(g,\Omega)<\frac12$ and that
the magnetic flow on the energy level $E^{-1}(\frac12)$ has no conjugate
points. Then
\[
\left|\int_M f\,dA_g\right|
\leq\frac{1}{2\pi}\int_{T_1M}\kappa\,d\mu_l ,
\]
where $\kappa(\theta)$ is the geodesic curvature at $\pi(\theta)$ of the
horocycle centred at the endpoint at infinity of the lifted unit-speed
magnetic geodesic with initial velocity $\theta$.
\end{maintheorem}

The normalization $E=\frac12$ is convenient but entails no loss of
generality. Indeed, let $e>c(g,\Omega)$ and set
\[
\Omega_e=\frac{\Omega}{\sqrt{2e}}
=\frac{f}{\sqrt{2e}}\,dA_g.
\]
Dividing a primitive of $\widetilde\Omega$ by $\sqrt{2e}$ gives one of
$\widetilde\Omega_e$, so
\[
c(g,\Omega_e)=\frac{c(g,\Omega)}{2e}<\frac12.
\]
Moreover, a magnetic geodesic of $(g,\Omega)$ on $E^{-1}(e)$ has constant
speed $\sqrt{2e}$. After parametrizing it by arclength, its equation becomes
\[
\nabla_{\dot\sigma}\dot\sigma
=
\frac{f}{\sqrt{2e}}\,i\dot\sigma,
\]
which is precisely the equation of a unit-speed magnetic geodesic for
$(g,\Omega_e)$ on $E^{-1}(\frac12)$. Thus, the magnetic geodesics at energy
$e$ for $(g,\Omega)$ and those at energy $\frac12$ for $(g,\Omega_e)$ have
the same trajectories, up to a constant reparametrization. In particular,
the two flows are conjugate by the dilation $v\mapsto\sqrt{2e}\,v$ followed
by a constant time change; since the dilation preserves the fibres of $TM$,
conjugate points correspond. Since the trajectories coincide, they also have
the same endpoints at infinity and the same associated horocycles.

For $\widetilde\theta\in T_1\widetilde M$, let
$\widetilde\sigma^e_{\widetilde\theta}$ be the lifted magnetic geodesic of
energy $e$ with initial velocity $\sqrt{2e}\,\widetilde\theta$, and let
$\kappa_e(\widetilde\theta)$ be the geodesic curvature at
$\pi(\widetilde\theta)$ of the horocycle centred at
$\widetilde\sigma^e_{\widetilde\theta}(+\infty)$. Since deck
transformations preserve $\widetilde g$ and $\widetilde\Omega$,
$\kappa_e(\widetilde\theta)$ depends only on the projection
$\theta\in T_1M$ of $\widetilde\theta$, and defines a function
$\kappa_e:T_1M\to\mathbb R$. Thus $\kappa_{1/2}=\kappa$.

\begin{corollary}\label{energia arbitraria}
Let $(M,g)$ be a surface without focal points, let $e>c(g,\Omega)$, and
suppose that the magnetic flow of $(g,\Omega)$ has no conjugate points on
$E^{-1}(e)$. Then
\[
\left|\int_M f\,dA_g\right|
\leq\frac{\sqrt{2e}}{2\pi}\int_{T_1M}\kappa_e\,d\mu_l .
\]
\end{corollary}

\begin{proof}
By the rescaling discussed above, $(g,\Omega_e)$ satisfies the hypotheses
of Theorem~\ref{Teo A}, its Lorentz force is $f/\sqrt{2e}$, and its
function $\kappa$ is $\kappa_e$. Applying Theorem~\ref{Teo A} to
$(g,\Omega_e)$ and multiplying the resulting inequality by $\sqrt{2e}$
gives the claim.
\end{proof}

By Lemma~\ref{horocycle and geodesic curvature} and Eberlein's comparison
estimate \cite[Lemma 2.8]{Eberlein1973}, the geodesic curvature of the
horocycles is bounded above by $K_0$ whenever $K_g\geq-K_0^2$. Since
\[
\mu_l(T_1M)=2\pi\operatorname{Area}_g(M),
\]
the previous corollary gives the following estimate at an arbitrary
energy level.

\begin{corollary}\label{cor curvatura}
Let $(M,g)$ be a surface without focal points, let $e>c(g,\Omega)$, and
suppose that the magnetic flow of $(g,\Omega)$ has no conjugate points on
$E^{-1}(e)$. If $K_g\geq-K_0^2$, then
\[
\left|\int_M f\,dA_g\right|
\leq
\sqrt{2e}\,K_0\operatorname{Area}_g(M).
\]
\end{corollary}

The two dynamical hypotheses of Theorem~\ref{Teo A} play complementary roles.
Supercriticality makes the lifted magnetic trajectories the geodesics of a
Randers metric on $\widetilde M$ which is bi-Lipschitz equivalent to
$\widetilde g$. The absence of conjugate points makes those geodesics globally
minimizing, hence quasi-geodesics of $(\widetilde M,\widetilde g)$, so that
each of them remains at bounded Hausdorff distance from a $\widetilde
g$-geodesic. The
absence of focal points enters separately, through the convexity of
horospheres.

The shadowing property just described is of independent interest, and yields
our second result.

\begin{maintheorem}\label{distinct endpoints}
Let $e>c(g,\Omega)$ and suppose that $(M,g)$ has no conjugate points.
Assume moreover that the magnetic flow has no conjugate points on
$E^{-1}(e)$. Then:
\begin{enumerate}
\item[(i)] every lifted magnetic geodesic of energy $e$ has two distinct
endpoints in $\widetilde M(\infty)$;
\item[(ii)] any two distinct points of $\widetilde M(\infty)$ are joined
by a minimal magnetic geodesic of energy $e$.
\end{enumerate}
\end{maintheorem}

This result extends the ideas of Peyerimhoff--Siburg
\cite{PS}. Their argument relies on negative
curvature, whereas here this assumption is replaced by the absence of
conjugate points on the magnetic energy level. The proof combines the
Randers reformulation of the magnetic flow with the classical Morse
argument.\\

Let us describe the method of proof of Theorem~\ref{Teo A}. Given a lifted
magnetic geodesic $\widetilde\sigma_{\widetilde\theta}$ and a $\widetilde
g$-geodesic $\gamma$ shadowing it, the Busemann function of $\gamma$ provides
an orthogonal coordinate system on $\widetilde M$ whose coordinate curves are
the horocycles and the orbits of the Busemann flow. Liouville's formula in
these coordinates expresses the geodesic curvature of
$\widetilde\sigma_{\widetilde\theta}$, which is $f$, by the identity above, as the derivative of the angle $\tau_{\widetilde\theta}$ between the
trajectory and the horocyclic foliation, plus a term involving the curvature
of the horocycles. Averaging along the flow and applying Birkhoff's ergodic
theorem, the derivative term contributes nothing provided
$\tau_{\widetilde\theta}$ stays bounded, and the inequality follows.\\

\section{Global geometry of surfaces}\label{Global geometry of surfaces}
Let $(M,g)$ be a closed connected Riemannian surface of genus at least $2$. We will denote by $TM$ the tangent bundle of $M$ and by $T_1M$ the unit tangent bundle. We also denote by $\pi: TM \rightarrow M$ the canonical projection. We denote by $(\widetilde M, \widetilde g)$ the universal covering of $M$ with the metric $\widetilde g = p^*g$, where $p:\widetilde M \rightarrow M$ is the covering map. 

The Levi-Civita connection $\nabla$ of $(M,g)$ determines the geodesics as the curves satisfying the equation $\nabla_{\gamma'(t)}\gamma'(t)=0$ for every $t \in \mathbb{R}$. For every $\theta=(x,v)\in TM\setminus\{0\}$, there exists a unique geodesic $\gamma_\theta:\mathbb{R}\to M$ such that $\gamma_\theta(0)=x$ and $\gamma_\theta'(0)=v$. Since the geodesics have constant speed, we shall assume that they have unit speed. The geodesic flow $\phi_t:T_1M \to T_1M$ of $(M,g)$ is given by $$\phi_t(\theta)=(\gamma_\theta(t),\gamma_\theta'(t))$$ for every $t\in\mathbb{R}$.

\subsection{Jacobi fields}\label{Jacobi fields} We recall that a Jacobi field along a geodesic $\gamma_\theta$ is a vector field satisfying
\[
J''(t) + R(J(t), \dot\gamma_\theta(t))\dot\gamma_\theta(t) = 0,
\]
where $R$ is the curvature tensor associated to $(M,g)$ and $'$ denotes covariant differentiation along $\gamma_\theta$. Two points $\gamma_\theta(a)$ and $\gamma_\theta(b)$ along the geodesic $\gamma_\theta$ are said to be \textit{conjugate} if there exists a nonzero Jacobi field $J$ along $\gamma_\theta$ such that $J(a) = J(b) = 0$. We say that $(M,g)$ has no conjugate points if every geodesic has no conjugate points.\\\\
\textbf{Convention.} We will denote by $e_\theta(t)$ the orthogonal vector field along the geodesic $\gamma_\theta$ such that $\{\gamma_\theta'(t), e_\theta(t)\}$ is positively oriented. If there is no risk of confusion, we will denote $e_\theta$ simply by $e$.\vspace{.5cm}

Let $J$ be an orthogonal Jacobi field along $\gamma_\theta$, then $J(t)=j(t)e(t)$, then $j$ satifies the scalar Jacobi equation $$j''(t)+k(t)j(t)=0,$$
where $k(t)=k(\gamma_\theta(t)).$ suposing that $j(t)\neq0$ in $t\in(a,b)$, the function $u(t)=j'(t)/j(t)$, satifies the Riccati equation $$u'(t)+u^2(t)+k(t)=0$$
for all $t \in (a,b).$

\subsection{Green bundles}

Suppose that $(M,g)$ has no conjugate points. Let $\theta=(p,v)\in TM$ and $w \in T_xM$ orthogonal to $v$. For $T\neq 0$, denote by $J_{T,w}$ the Jacobi field along $\gamma_\theta$ satisfying
$$
J_{T,w}(0)=w, \qquad J_{T,w}(T)=0.
$$
It was proved in \cite{Green1958} that the limits
$$
J^s_\theta(t)=\lim_{T\to +\infty} J_{T,w}(t), 
\qquad 
J^u_\theta(t)=\lim_{T\to -\infty} J_{T,w}(t)
$$
exist for every $t \in \mathbb{R}$. The fields $J^s_\theta$ and $J^u_\theta$ are called the \textit{stable} and \textit{unstable} Jacobi fields, respectively. They are everywhere orthogonal to $\dot{\gamma}_\theta$ and nowhere vanishing. Consequently, they determine solutions of the Riccati equation defined for all $t\in\mathbb{R}$ by
$$
u^s_\theta(t)=\frac{(j^s_\theta)'(t)}{j^s_\theta(t)}, 
\qquad 
u^u_\theta(t)=\frac{(j^u_\theta)'(t)}{j^u_\theta(t)},
$$
where the decompositions
$$
J^s_\theta(t)=j^s_\theta(t)e(t), 
\qquad 
J^u_\theta(t)=j^u_\theta(t)e(t)
$$
are taken with respect to the unit vector field $e(t)$ orthogonal to $\dot{\gamma}_\theta$ such that $\{\gamma'_\theta(t),e(t)\}$ is well oriented. In particular $u^{\sigma}_\theta(t+s)=u^\sigma_{\phi_s(\theta)}(t)$ for all $t,s \in \mathbb R$ and $\sigma\in\{s,u\}$.
   \subsection{Busemann functions}  
   
   \begin{assumption}
    From now on, we assume that $(M,g)$ has no conjugate points. 
\end{assumption}

    Given $\widetilde\theta\in T_1\widetilde M$, \textit{the forward Busemann function} $b_{\widetilde \theta}^+:\widetilde M \rightarrow \mathbb R$ associated to $\widetilde \theta$ is defined by 
$$b_{ \widetilde \theta}^+(x):=\lim_{t \to +\infty}(d_g(x,\gamma_{\widetilde \theta}(t))-t).$$

The function $b_{\widetilde{\theta}}^+$ is $C^{1,r}$ for every $\widetilde{\theta}\in T_1\widetilde{M}$, where $C^{1,r}$ means that the function is continuously differentiable and its derivative is $r$-Lipschitz continuous.

 We  call the level sets of \( b_{\widetilde \theta} \) \textit{positive horospheres} and denote them by \( H_{\widetilde \theta}^+(s) = (b^+_{\widetilde \theta})^{-1}\{-s\} \). We also define the \textit{negative horospheres} by \( H^-_{\widetilde \theta}(t) = H^+_{-\widetilde \theta}(t) \). The integral curves of the flow $-\nabla_p b_{\widetilde \theta}^+$, $\lambda^{\widetilde \theta}_t:\widetilde M \to \widetilde M$, are geodesics orthogonal to the horospheres $H^+_{\widetilde \theta}$. In particular, the geodesic $\widetilde \gamma_{\widetilde \theta}$ is an orbit of this flow and we have that 
    \[
    \lambda^{\widetilde \theta}_t(H^+_{\widetilde \theta}(s))=H^+_{\widetilde \theta}(s+t)
    \]

for every $t,s \in \mathbb{R}.$For details concerning Busemann functions and horospheres, see \cite{Eschenburg1977} and \cite{Pesin1977}.

Given two subsets $A,B\subset\widetilde M$, their \textit{Hausdorff
distance} is
\[
d_{H,\widetilde g}(A,B)=\max\Bigl\{
\sup_{a\in A}d_{\widetilde g}(a,B),\ \sup_{b\in B}d_{\widetilde g}(b,A)
\Bigr\} .
\]
For curves, $d_{H,\widetilde g}$ always refers to their images; it does not
compare points with the same parameter, in contrast with the following
notion.

Two geodesics $\gamma_1$ and $\gamma_2$ in $\widetilde{M}$ are \emph{asymptotic for $t>0$} (resp. for $t<0$) if there exists $D>0$ such that
\[
d(\gamma_1(t),\gamma_2(t)) \le D \quad \text{for all } t>0 \ (\text{resp. } t<0).
\]
When not specified, ``asymptotic'' means asymptotic for $t>0$. Two
geodesics are said to be asymptotic if their distance remains bounded as
$t\to+\infty$. An equivalence class of asymptotic geodesics is called a
point at infinity. This relation is an equivalence relation. The
equivalence class of $\gamma$ for $t>0$ (resp. $t<0$) is denoted by
$\gamma(+\infty)$ (resp. $\gamma(-\infty)$), and the set of all such
classes is denoted by $\widetilde{M}(\infty)$.\\
\noindent \textbf{Convention.} For all $\widetilde{\theta}\in T
_1\widetilde{M}$, the curve $h^+_{\widetilde{\theta}}:\mathbb{R}\to \widetilde{M}$ is the arc-length parametrization of $H^+_{\widetilde{\theta}}$ such that $\{(h^+)'_{\widetilde{\theta}}(s), -\nabla b_{\widetilde{\theta}}(h^+_{\widetilde{\theta}}(s)) \}$ has the canonical orientation.

\subsection{Surfaces without focal points and flat strips}
We recall that a surface $(M,g)$ has no focal points if, for every geodesic $\gamma : \mathbb{R} \to M$ and every nontrivial Jacobi field $J$ along $\gamma$ satisfying $J(0)=0$, the function $t \mapsto |J(t)|^2$ is strictly increasing for $t>0$. In particular, every surface without focal points has no conjugate points.

\begin{lemma}\label{faixa plana}
Let $\gamma_{\widetilde\theta}$ be a geodesic of
$(\widetilde M,\widetilde g)$, let $[a,b]\subset\mathbb R$ be a compact
interval and let $D>0$. Then there exists
$s_0=s_0(a,b,D)>0$ such that for every $s\geq s_0$ the asymptotic geodesic
\[
\gamma_\eta(t)=\lambda^{\widetilde\theta}_t
\bigl(h^+_{\widetilde\theta}(s)\bigr)
\]
satisfies
\[
\gamma_\eta[a,b]\cap V_D\bigl(\gamma_{\widetilde\theta}[a,b]\bigr)=\emptyset ,
\]
where $V_D\bigl(\gamma_{\widetilde\theta}[a,b]\bigr)$ denotes the tubular
neighbourhood of radius $D$ of $\gamma_{\widetilde\theta}[a,b]$.
\end{lemma}

\begin{proof}
The curve $h^+_b(s)=\lambda^{\widetilde\theta}_b
\bigl(h^+_{\widetilde\theta}(s)\bigr)$ parametrizes the horocycle
$H^+_{\widetilde\theta}(b)$, and it is proper, so
$d\bigl(h^+_b(s),\gamma_{\widetilde\theta}(b)\bigr)\to\infty$ as
$s\to\infty$. Choose $s_0>0$ such that
\[
d\bigl(h^+_b(s),\gamma_{\widetilde\theta}(b)\bigr)>D
\qquad\text{for every }s\geq s_0 ,
\]
and fix $s\geq s_0$. Since $\gamma_{\widetilde\theta}$ and $\gamma_\eta$
are asymptotic, the absence of focal points implies that
$t\mapsto d\bigl(\gamma_{\widetilde\theta}(t),\gamma_\eta(t)\bigr)$ is
non-increasing; as $\gamma_\eta(b)=h^+_b(s)$, we get
\[
d\bigl(\gamma_{\widetilde\theta}(t),\gamma_\eta(t)\bigr)
\geq d\bigl(\gamma_{\widetilde\theta}(b),\gamma_\eta(b)\bigr)>D
\qquad\text{for every }t\in[a,b].
\]
In particular $\gamma_\eta[a,b]\cap
V_D\bigl(\gamma_{\widetilde\theta}[a,b]\bigr)=\emptyset$.
\end{proof}
\subsection{Geodesic curvature of horocycles} Let $c:(a,b) \to M$ be an embedded $C^2$ curve parametrized by arc length. The \textit{geodesic curvature} of $c$, denoted by $k_c$, at the point $c(t)$ is given by
\[
k_c(t) = g\big(\nabla_{c'(t)} c'(t),\, i\,c'(t)\big).
\]
where $ic$ is such that $\{c'(t),ic'(t)\}$ is positively oriented.


\begin{lemma}\label{horocycle and geodesic curvature} The geodesic curvature $k_{h_\theta^+}(0)$ of the horocycle $h_\theta^+$ at $t=0$  coincide with $-u_\theta^s(0).$
    \end{lemma}
    \begin{proof} Suppose that $J^s_\theta(t)=j_\theta^s(t)e(t)$ with $e$ as in section \ref{Jacobi fields}.We can also assume that $j^s_\theta(0)=1.$ Consider the variation $f:(-\varepsilon,\varepsilon)\times (-\varepsilon,\varepsilon) \to \widetilde{M}$ for $\varepsilon>0$, $f(s,t)=\lambda^{\widetilde \theta}_t(h^+_{\widetilde \theta}(s))$.

       \begin{align*}
           k_{h_\theta^+}(0)=&\langle \nabla_ {(h^+_{\widetilde \theta})'} (h^+_{\widetilde \theta})'(0),-\gamma_{\widetilde \theta}'(0)\rangle\\
           =& \langle\nabla_{\frac{\partial f}{\partial s}(0,0)}\frac{\partial f}{\partial s}(0,0),-\frac{\partial f}{\partial t}(0,0)\rangle\\
           =& \langle\frac{\partial f}{\partial s}(0,0),\nabla_{\frac{\partial f}{\partial s}(0,0)}\frac{\partial f}{\partial t}(0,0)\rangle\\
            =& \langle\frac{\partial f}{\partial s}(0,0),\nabla_{\frac{\partial f}{\partial t}(0,0)}\frac{\partial f}{\partial s}(0,0)\rangle\\
            =&\langle -e(0),(j^s_\theta)'(0)e(0)\rangle\\
            =&-(j_\theta^s)'(0)
        \end{align*}
        where we used that $\nabla_{\frac{\partial f}{\partial t}(0,0)}\frac{\partial f}{\partial s}(0,0) = (J^s_\theta)'(0)$ (see Proposition 3.1 in \cite{Heintze1977} for a proof).
    \end{proof}
    \begin{corollary}\label{horociclo curvatura no-negativa}
        Let $(M,g)$ be a surface without focal points. Then the geodesic curvature of horocycles is nonnegative.
    \end{corollary}
    \begin{proof}
       It is easy to see that, in the absence of focal points, the stable Riccati solution is nonpositive, and the result follows from Lemma \ref{horocycle and geodesic curvature}.
    \end{proof}

\section{Magnetic flows and geodesic curvature}
\label{Magnetic flows and geodesic curvature}

Throughout this section $(M,g)$ is a closed oriented Riemannian surface of
genus greater than one and without conjugate points, and $\Omega$ is a closed
$2$-form on $M$. Let $p:\widetilde M\to M$ be the universal covering map,
and write $\widetilde g=p^*g$ and $\widetilde\Omega=p^*\Omega$. Under these
hypotheses $(\widetilde M,\widetilde g)$ is a uniform visibility manifold
and is quasi-isometric to the hyperbolic plane
\cite{EberleinONeill1973,GomesRuggiero2013}.

Consider the twisted symplectic form
\[
\omega_\Omega=\omega_{\mathrm{can}}+\pi^*\Omega.
\]
The Hamiltonian $E(x,v)=\frac12\|v\|^2$ generates a flow whose trajectories
satisfy
\begin{equation}\label{mag eq}
\nabla_{\dot\sigma}\dot\sigma=Y(\dot\sigma),
\end{equation}
where $\Omega(u,v)=g(Y(u),v)$. Since $\dim M=2$, there is a unique
$f\in C^\infty(M)$ such that $\Omega=f\,dA_g$, and hence $Y=f\,i$, where
$i$ denotes rotation by $\pi/2$ in the positive direction.

\begin{lemma}\label{curvatura=lorentz}
Let $\sigma$ be a magnetic geodesic with $\|\dot\sigma\|\equiv1$. Then
$k_\sigma(t)=f(\sigma(t))$ for every $t$.
\end{lemma}

\begin{proof}
By \eqref{mag eq},
\[
k_\sigma
=g(\nabla_{\dot\sigma}\dot\sigma,i\dot\sigma)
=g(Y(\dot\sigma),i\dot\sigma)
=\Omega(\dot\sigma,i\dot\sigma)
=f\,dA_g(\dot\sigma,i\dot\sigma)=f.
\]
\end{proof}

Since $\widetilde M$ is simply connected, $\widetilde\Omega=d\widetilde\lambda$
for some $1$-form $\widetilde\lambda$. Thus the lifted magnetic flow is the
Euler--Lagrange flow of
\[
L(x,v)=\frac12\widetilde g_x(v,v)+\widetilde\lambda_x(v),
\]
and its Mañé critical value is
\[
c(g,\Omega)=
\inf_{u\in C^\infty(\widetilde M,\mathbb R)}
\sup_{x\in\widetilde M}
\frac12\|d_xu+\widetilde\lambda_x\|_{\widetilde g}^2,
\]
see \cite{BP2002} for more details.

Fix $e>c(g,\Omega)$ and choose $u\in C^\infty(\widetilde M)$ such that
\[
c_0:=\sup_{x\in\widetilde M}
\|d_xu+\widetilde\lambda_x\|_{\widetilde g}<\sqrt{2e}.
\]
Then
\[
F_e(x,v)=\sqrt{2e\,\widetilde g_x(v,v)}
+(d_xu+\widetilde\lambda_x)(v)
\]
is a Randers metric; see \cite{BCS} for more details on
Randers and Finsler metrics.. Moreover,
\begin{equation}\label{bilipschitz}
(\sqrt{2e}-c_0)\|v\|_{\widetilde g}
\leq F_e(x,v)\leq
(\sqrt{2e}+c_0)\|v\|_{\widetilde g}.
\end{equation}
Thus $d_{F_e}$ and $d_{\widetilde g}$ are bi-Lipschitz equivalent, and
the lifted magnetic trajectories of energy $e$ are, up to reparametrization,
the geodesics of $F_e$.

We now assume that the magnetic flow has no conjugate points on
$E^{-1}(e)$. The following standard consequence will be used.

\begin{lemma}\label{minimizantes}
The Randers metric $F_e$ has no conjugate points, and every $F_e$-geodesic
is globally minimizing in $\widetilde M$.
\end{lemma}

\begin{proof}
The magnetic flow on $E^{-1}(e)$ and the geodesic flow of $F_e$ are
conjugate up to reparametrization. Hence $F_e$ has no conjugate points.
The minimizing property follows from
\cite[Corollary~4.2]{GomesRuggiero2013}.
\end{proof}

We shall use the supercritical Morse argument of
Peyerimhoff--Siburg \cite[Theorem~2.9]{PS}. Namely, by
\eqref{bilipschitz} and Lemma~\ref{minimizantes}, every $F_e$-geodesic is
a uniform quasi-geodesic of $(\widetilde M,\widetilde g)$. Since $M$ has
genus greater than one, we may choose a metric $g_-$ of constant negative
curvature on $M$. Its lift $\widetilde g_-$ is negatively curved and is
uniformly equivalent to $\widetilde g$. The classical Morse lemma applied
to $\widetilde g_-$ therefore gives the uniform shadowing estimate used
below.

\begin{lemma}[Morse]\label{morse}
There exists $D>0$ such that every $F_e$-minimizing segment is at
$d_{\widetilde g}$-Hausdorff distance less than $D$ from the corresponding
$\widetilde g$-geodesic segment. Moreover, every $F_e$-geodesic is at
$d_{\widetilde g}$-Hausdorff distance less than $D$ from a
$\widetilde g$-geodesic.
\end{lemma}

\begin{proof}
This is the supercritical Morse argument of
\cite[Theorem~2.9]{PS}, with the negatively curved reference metric
$\widetilde g_-$ above. The uniform equivalence of $g$ and $g_-$ transfers
the resulting estimates to $\widetilde g$.
\end{proof}

\begin{proof}[Proof of Theorem \ref{distinct endpoints}]
Let $\widetilde\sigma:\mathbb R\to\widetilde M$ be a lifted magnetic
geodesic of energy $e$. By Lemma~\ref{minimizantes} it is
$F_e$-minimizing, and hence Lemma~\ref{morse} gives a
$\widetilde g$-geodesic $\gamma$ with
$d_{H,\widetilde g}(\widetilde\sigma,\gamma)<D$. Since $\gamma$ has two
distinct endpoints in $\widetilde M(\infty)$, so does
$\widetilde\sigma$. This proves (i).

For (ii), let $\xi,\eta\in\widetilde M(\infty)$ be distinct. By visibility,
choose a $\widetilde g$-geodesic $c$ with
$c(-\infty)=\xi$ and $c(+\infty)=\eta$. Set $p_n=c(-n)$ and $q_n=c(n)$,
and let $\sigma_n$ be an $F_e$-geodesic joining $p_n$ to $q_n$.
By Lemma~\ref{morse}, the $\sigma_n$ remain uniformly close to the
corresponding geodesic segments. Hence they meet a fixed compact set.
After passing to a subsequence, their initial conditions converge, and
continuous dependence gives a complete magnetic geodesic $\sigma$.
The shadowing estimate gives
\[
\sigma(-\infty)=\xi,\qquad \sigma(+\infty)=\eta.
\]
This proves (ii).
\end{proof}
\section{Liouville's formula and horocyclic coordinates}\label{Liouville's formula and horocyclic coordinates}

We want to establish a relationship between the Lorentz force of the
magnetic system and the geodesic curvature of a family of horocycles of
$(M,g)$.

\begin{lemma}[Liouville]\label{liouville}
Let $(M,g)$ be a $C^{\infty}$ Riemannian surface and let
$\Phi: U \longrightarrow V$ be an oriented local coordinate chart, whose
coordinate vector fields

$$
X=\frac{\partial}{\partial x}\Phi(x,y),
\qquad
Y=\frac{\partial}{\partial y}\Phi(x,y)
$$

are orthogonal. Let $c:(a,b)\longrightarrow V$ be a $C^2$ curve, and let
$\tau(t)$ be the oriented angle formed by $c'(t)$ and $X(c(t))$. Then the
geodesic curvature $k_c(t)$ of $c$ at the point $c(t)$ satisfies

$$
k_c(t)
=
\frac{d\tau}{dt}(t)
+
k_X(c(t))\cos(\tau(t))
+
k_Y(c(t))\sin(\tau(t)),
$$

where $k_X(c(t))$ and $k_Y(c(t))$ are, respectively, the geodesic
curvatures of the integral curves of $X$ and $Y$ at the point $c(t)$.
\end{lemma}

Given $\widetilde\theta\in T_1\widetilde M$, let
$\widetilde\sigma_{\widetilde\theta}$ denote the corresponding
unit-speed magnetic geodesic. Let $\gamma$ be a $\widetilde g$-geodesic
shadowing $\widetilde\sigma_{\widetilde\theta}$. Denote by
$H_\gamma^+$ the horocycle determined by $\gamma$, and let

$$
h_\gamma^+:\mathbb R\longrightarrow H_\gamma^+
$$

be its unit-speed parametrization satisfying

$$
h_\gamma^+(0)=\gamma(0),
$$

with the orientation induced by the orientation of $\widetilde M$.

Let $\lambda_t^\gamma$ denote the Busemann flow associated with $\gamma$.
We define the horocyclic parametrization

$$
\Phi_{\widetilde\theta,\gamma}:\mathbb R^2\longrightarrow\widetilde M
$$

by

$$
\Phi_{\widetilde\theta,\gamma}(t,r)
=
\lambda_t^\gamma
\bigl(h_\gamma^+(r)\bigr).
$$

The construction above depends on the choice of the shadow $\gamma$ at
the level of the parametrization, but the associated horocyclic
foliation does not. Indeed, suppose that $\gamma_1$ and $\gamma_2$ are
two $\widetilde g$-geodesics shadowing the same magnetic geodesic
$\widetilde\sigma_{\widetilde\theta}$. Since both shadows have the same
endpoint at infinity,

$$
\gamma_1(+\infty)=\gamma_2(+\infty),
$$

their Busemann functions differ by a constant: $
b_{\gamma_1}=b_{\gamma_2}+C.
$ Consequently, $
\nabla b_{\gamma_1}
=
\nabla b_{\gamma_2},
$ and hence the corresponding Busemann flows and horocyclic foliations
coincide.

The parametrizations $h_{\gamma_1}^+$ and $h_{\gamma_2}^+$ may differ by
a translation of the parameter. Thus the coordinate systems
$\Phi_{\widetilde\theta,\gamma_1}$ and
$\Phi_{\widetilde\theta,\gamma_2}$ may differ by a translation in the
second coordinate. However, their coordinate vector fields define the
same oriented orthogonal frame on $\widetilde M$.

For a fixed shadow $\gamma$, let

$$
X_{\widetilde\theta}
=
\frac{\partial}{\partial t}
\Phi_{\widetilde\theta,\gamma},
\qquad
Y_{\widetilde\theta}
=
\frac{\partial}{\partial r}
\Phi_{\widetilde\theta,\gamma}.
$$
The vector field $X_{\widetilde\theta}$ is tangent to the trajectories of
the Busemann flow, while $Y_{\widetilde\theta}$ is tangent to the
horocycles. Since the trajectories of the Busemann flow are geodesics,
the geodesic curvature of the integral curves of
$X_{\widetilde\theta}$ vanishes.

Let $c:\mathbb R\longrightarrow\widetilde M$ be a regular embedded
$C^2$ curve. Applying Lemma~\ref{liouville} to the coordinate system
$\Phi_{\widetilde\theta,\gamma}$, we obtain

$$
k_c(t)
=
\frac{d\tau}{dt}(t)
+
k_{H_\gamma^+}\bigl(c(t)\bigr)
\sin\bigl(\tau(t)\bigr),
$$

where $\tau:\mathbb R\to\mathbb R$ is a continuous lift of the oriented
angle between $c'(t)$ and
$X_{\widetilde\theta}(c(t))$.

We now apply this formula to the magnetic geodesic
$c=\widetilde\sigma_{\widetilde\theta}$. We denote by
$\tau_{\widetilde\theta}:\mathbb R\to\mathbb R$ the corresponding
continuous lift, normalized by

$$
\tau_{\widetilde\theta}(0)
=
\angle\left(
\widetilde\theta,
-\nabla b_\gamma\bigl(\pi(\widetilde\theta)\bigr)
\right).
$$

We also define

$$
k_{\widetilde\theta}(t)
:=
k_{H_\gamma^+}
\bigl(\widetilde\sigma_{\widetilde\theta}(t)\bigr).
$$

Therefore, along $\widetilde\sigma_{\widetilde\theta}$,
Liouville's formula takes the form

$$
k_{\widetilde\sigma_{\widetilde\theta}}(t)
=
\frac{d\tau_{\widetilde\theta}}{dt}(t)
+
k_{\widetilde\theta}(t)
\sin\bigl(\tau_{\widetilde\theta}(t)\bigr).
$$

The quantities $k_{\widetilde\theta}(t)$ and
$\tau_{\widetilde\theta}(t)$ do not depend on the choice of the shadow
$\gamma$. Indeed, if $\gamma_1$ and $\gamma_2$ are two shadows of
$\widetilde\sigma_{\widetilde\theta}$, then

$$
b_{\gamma_1}=b_{\gamma_2}+C
$$

for some constant $C$. Hence

$$
\nabla b_{\gamma_1}
=
\nabla b_{\gamma_2},
$$

and the corresponding horocyclic foliations coincide. Therefore,

$$
k_{H_{\gamma_1}^+}
\bigl(\widetilde\sigma_{\widetilde\theta}(t)\bigr)
=
k_{H_{\gamma_2}^+}
\bigl(\widetilde\sigma_{\widetilde\theta}(t)\bigr)
$$

for every $t\in\mathbb R$. Moreover,

$$
\angle\left(
\widetilde\theta,
-\nabla b_{\gamma_1}
\bigl(\pi(\widetilde\theta)\bigr)
\right)
=
\angle\left(
\widetilde\theta,
-\nabla b_{\gamma_2}
\bigl(\pi(\widetilde\theta)\bigr)
\right),
$$

so the normalization of $\tau_{\widetilde\theta}(0)$ is also
independent of the choice of $\gamma$. Consequently,
$k_{\widetilde\theta}$ and $\tau_{\widetilde\theta}$ are well-defined.

We next study the behavior of the quantities
$k_{\widetilde\theta}(t)$ and $\tau_{\widetilde\theta}(t)$ under deck
transformations.

\begin{lemma}\label{horocyclic-isometry}
Let $\sigma_\theta$ be a unit-speed magnetic geodesic and let
$\widetilde\sigma_{\widetilde\theta_1}$ and
$\widetilde\sigma_{\widetilde\theta_2}$ be two lifts of
$\sigma_\theta$. Then there exists a deck transformation
$T:\widetilde M\longrightarrow\widetilde M$ such that

$$
\widetilde\sigma_{\widetilde\theta_2}
=
T\circ\widetilde\sigma_{\widetilde\theta_1}.
$$

Moreover, if $\gamma_1$ is a shadow of
$\widetilde\sigma_{\widetilde\theta_1}$, then

$$
\gamma_2=T\circ\gamma_1
$$

is a shadow of $\widetilde\sigma_{\widetilde\theta_2}$.
\end{lemma}
\begin{proof}
Since $\widetilde\sigma_{\widetilde\theta_1}$ and
$\widetilde\sigma_{\widetilde\theta_2}$ are lifts of the same curve
$\sigma_\theta$, we have
$p\circ\widetilde\sigma_{\widetilde\theta_1}
=p\circ\widetilde\sigma_{\widetilde\theta_2}=\sigma_\theta$. The deck
transformation group acts transitively on the fibres of $p$, so there is a
deck transformation $T$ with
$T\bigl(\widetilde\sigma_{\widetilde\theta_1}(0)\bigr)
=\widetilde\sigma_{\widetilde\theta_2}(0)$ and
$dT\bigl(\dot{\widetilde\sigma}_{\widetilde\theta_1}(0)\bigr)
=\dot{\widetilde\sigma}_{\widetilde\theta_2}(0)$. Since $T$ preserves
$\widetilde g$ and $\widetilde\Omega$, it preserves the magnetic equation
\eqref{mag eq}, so $T\circ\widetilde\sigma_{\widetilde\theta_1}$ is a
magnetic geodesic with the same initial condition as
$\widetilde\sigma_{\widetilde\theta_2}$; by uniqueness of solutions,
\[
\widetilde\sigma_{\widetilde\theta_2}=T\circ\widetilde\sigma_{\widetilde\theta_1}.
\]

Let now $\gamma_1$ be a shadow of $\widetilde\sigma_{\widetilde\theta_1}$
and set $\gamma_2=T\circ\gamma_1$, which is a $\widetilde g$-geodesic since
$T$ is a $\widetilde g$-isometry. Isometries preserve Hausdorff distances,
so
\[
d_{H,\widetilde g}\bigl(\gamma_2,\widetilde\sigma_{\widetilde\theta_2}\bigr)
=d_{H,\widetilde g}\bigl(T\gamma_1,T\widetilde\sigma_{\widetilde\theta_1}\bigr)
=d_{H,\widetilde g}\bigl(\gamma_1,\widetilde\sigma_{\widetilde\theta_1}\bigr)<D ,
\]
and $\gamma_2$ is a shadow of $\widetilde\sigma_{\widetilde\theta_2}$.

Finally, from the definition of the Busemann function and the fact that $T$
is an isometry with $T\circ\gamma_1=\gamma_2$ we get
$b_{\gamma_2}\circ T=b_{\gamma_1}$; in particular $T$ maps the horocyclic
foliation of $\gamma_1$ onto that of $\gamma_2$.
\end{proof}

\begin{lemma}\label{invariance-k-tau}
Let $T:\widetilde M\longrightarrow\widetilde M$ be a deck
transformation. Then, for every $\widetilde\theta\in T_1\widetilde M$
and every $t\in\mathbb R$,
\[
k_{dT(\widetilde\theta)}(t)
=
k_{\widetilde\theta}(t)
\]
and
\[
\tau_{dT(\widetilde\theta)}(t)
=
\tau_{\widetilde\theta}(t).
\]
\end{lemma}

\begin{proof}
Let $\gamma$ be a shadow of $\widetilde\sigma_{\widetilde\theta}$. By
Lemma~\ref{horocyclic-isometry},
\[
\widetilde\sigma_{dT(\widetilde\theta)}=T\circ\widetilde\sigma_{\widetilde\theta},
\qquad
b_{T\circ\gamma}\circ T=b_\gamma ,
\]
and $T\circ\gamma$ is a shadow of $\widetilde\sigma_{dT(\widetilde\theta)}$.

Since $T$ is a $\widetilde g$-isometry, it maps $H^+_\gamma$ onto
$H^+_{T\circ\gamma}$ and preserves geodesic curvature, so
\[
k_{dT(\widetilde\theta)}(t)
=k_{H^+_{T\circ\gamma}}\bigl(T(\widetilde\sigma_{\widetilde\theta}(t))\bigr)
=k_{H^+_\gamma}\bigl(\widetilde\sigma_{\widetilde\theta}(t)\bigr)
=k_{\widetilde\theta}(t).
\]

For the angle, differentiating $b_{T\circ\gamma}\circ T=b_\gamma$ gives
$\nabla b_{T\circ\gamma}(T(p))=dT\bigl(\nabla b_\gamma(p)\bigr)$. Taking
$p=\widetilde\sigma_{\widetilde\theta}(t)$ and using that $dT$ is a linear
isometry preserving the orientation, the oriented angle between
$\dot{\widetilde\sigma}_{\widetilde\theta}(t)$ and
$-\nabla b_\gamma$ at $p$ equals the one between
$\dot{\widetilde\sigma}_{dT(\widetilde\theta)}(t)$ and
$-\nabla b_{T\circ\gamma}$ at $T(p)$. Hence
\[
\tau_{dT(\widetilde\theta)}(t)\equiv\tau_{\widetilde\theta}(t)
\pmod{2\pi},
\]
and since both are continuous lifts agreeing at $t=0$ by the normalization
of $\tau$, they agree for every $t\in\mathbb R$.
\end{proof}

\begin{corollary}\label{k-em-M}
The functions $k_{\widetilde\theta}(0)$ and
$\tau_{\widetilde\theta}(0)$ depend only on the projection
$\theta\in T_1M$ of $\widetilde\theta$. Hence they define
well-defined functions

$$
k,\tau:T_1M\longrightarrow\mathbb R
$$

given by

$$
k(\theta)
=
k_{\widetilde\theta}(0)
$$

and

$$
\tau(\theta)
=
\tau_{\widetilde\theta}(0),
$$

where $\widetilde\theta$ is any lift of $\theta$.
\end{corollary}

\section{Key lemmas and Proof of Theorem A} \label{Key lemmas and Proof of Theorem A}

Fix $\widetilde\theta\in T_1\widetilde M$ and let $\gamma$ be a
$\widetilde g$-geodesic shadowing $\widetilde\sigma_{\widetilde\theta}$,
say
\begin{equation}\label{shadow}
d_{H,\widetilde g}\bigl(\widetilde\sigma_{\widetilde\theta},\gamma\bigr)<D,
\end{equation}
with $D$ the constant of Lemma~\ref{morse}. Since $b_\gamma$ is defined up
to an additive constant --- equivalently, since $\gamma$ may be translated
in its parameter --- we normalize it by
\[
b_\gamma\bigl(\widetilde\sigma_{\widetilde\theta}(0)\bigr)=0 .
\]
This does \emph{not} mean that $\gamma$ meets
$\widetilde\sigma_{\widetilde\theta}$: it only places
$\widetilde\sigma_{\widetilde\theta}(0)$ on the horocycle $H^+_\gamma(0)$.

Because $b_\gamma$ is $C^1$ with $\|\nabla b_\gamma\|\equiv1$, along any
orbit of the Busemann flow $\lambda^\gamma_t$ (the flow of
$-\nabla b_\gamma$) we have
$\frac{d}{dt}b_\gamma\bigl(\lambda^\gamma_t(p)\bigr)=-1$, whence
\begin{equation}\label{busemann coord}
b_\gamma\bigl(\Phi_{\widetilde\theta,\gamma}(t,r)\bigr)=-t
\qquad\text{and}\qquad
b_\gamma\bigl(\gamma(s)\bigr)=-s .
\end{equation}
Moreover $\Phi_{\widetilde\theta,\gamma}$ is a homeomorphism of
$\mathbb R^2$ onto $\widetilde M$: the horocycles are properly embedded
copies of $\mathbb R$ foliating $\widetilde M$, and $\lambda^\gamma_t$ maps
$H^+_\gamma(s)$ diffeomorphically onto $H^+_\gamma(s+t)$.

We write
\[
\widetilde\sigma_{\widetilde\theta}(l)
=\Phi_{\widetilde\theta,\gamma}\bigl(u(l),v(l)\bigr),\qquad l\in\mathbb R ,
\]
so that by \eqref{busemann coord}
\[
u(l)=-b_\gamma\bigl(\widetilde\sigma_{\widetilde\theta}(l)\bigr),
\qquad u(0)=0 .
\]
The value $v(0)$ is an arbitrary real number. Since $b_\gamma$ is
$1$-Lipschitz and $\widetilde\sigma_{\widetilde\theta}$ has unit
$\widetilde g$-speed, $u$ is $1$-Lipschitz:
\begin{equation}\label{lipschitz u}
|u(l)-u(l')|\leq|l-l'|\qquad\text{for all }l,l'\in\mathbb R .
\end{equation}
Finally, if $d_{\widetilde g}\bigl(\widetilde\sigma_{\widetilde\theta}(l),
\gamma(s)\bigr)<D$, then \eqref{busemann coord} gives
\begin{equation}\label{u vs s}
\bigl|u(l)-s\bigr|<D .
\end{equation}


\begin{lemma}\label{u propia}
Under the standing hypotheses, $u(l)\to+\infty$ as $l\to+\infty$.
Consequently, for every $\ell\geq0$ the set
\[
I_\ell:=\bigl\{\,l\geq0:\ u(l)=\ell\,\bigr\}
\]
is nonempty and compact; equivalently,
$\widetilde\sigma_{\widetilde\theta}[0,\infty)$ meets every horocycle
$H^+_\gamma(\ell)$, $\ell\geq0$, in a nonempty compact set.
\end{lemma}

\begin{proof}
By \eqref{shadow}, for every $l$ there is $s(l)$ with
$d_{\widetilde g}\bigl(\widetilde\sigma_{\widetilde\theta}(l),\gamma(s(l))\bigr)<D$,
so \eqref{u vs s} gives $|u(l)-s(l)|<D$. By Lemma~\ref{minimizantes} the
curve $\widetilde\sigma_{\widetilde\theta}$ is $F_e$-minimizing, hence by
\eqref{bilipschitz}
\[
d_{\widetilde g}\bigl(\widetilde\sigma_{\widetilde\theta}(0),
\widetilde\sigma_{\widetilde\theta}(l)\bigr)
\geq\bigl(\sqrt{2e}+c_0\bigr)^{-1}
L_{F_e}\bigl(\widetilde\sigma_{\widetilde\theta}|_{[0,l]}\bigr)
\geq\frac{\sqrt{2e}-c_0}{\sqrt{2e}+c_0}\,l
\xrightarrow[l\to\infty]{}\infty .
\]
Therefore $|s(l)|\to\infty$. Since $\widetilde\sigma_{\widetilde\theta}$
and $\gamma$ are at finite Hausdorff distance, they have the same endpoint
at $+\infty$, so $s(l)\to+\infty$ as $l\to+\infty$, and $|u(l)-s(l)|<D$
gives $u(l)\to+\infty$.

As $u$ is continuous with $u(0)=0$ and $u(l)\to+\infty$, it attains every
value $\ell\geq0$ on $[0,\infty)$, so $I_\ell\neq\emptyset$; and $I_\ell$ is
closed and bounded because $u$ is proper on $[0,\infty)$.
\end{proof}


\begin{lemma}\label{ultima salida}
For every $\ell\geq0$ there is a unique $t_\ell\in I_\ell$ with
\[
v(t_\ell)=\max\bigl\{\,v(l):\ l\in I_\ell\,\bigr\}=:\rho_\ell .
\]
Moreover $t_\ell\geq\ell$; in particular $t_\ell>t_0$ whenever $\ell>t_0$.
\end{lemma}

\begin{proof}
The maximum exists by Lemma~\ref{u propia} and the continuity of $v$.
Uniqueness follows from the injectivity of $\Phi_{\widetilde\theta,\gamma}$
and of $\widetilde\sigma_{\widetilde\theta}$ (the latter is $F_e$-minimizing
by Lemma~\ref{minimizantes}): the point
$\Phi_{\widetilde\theta,\gamma}(\ell,\rho_\ell)$ determines its parameter.

By \eqref{lipschitz u} with $l'=0$ and $u(0)=0$,
\[
\ell=u(t_\ell)\leq t_\ell .\qedhere
\]
\end{proof}

We need the following classical result of differential geometry: 

\begin{lemma}[Hopf]\label{Hopf's Lemma}
Let $\Phi: U \subset \mathbb{R}^2 \to M$ be a parametrization of $M$, and
let $c: [a,b] \rightarrow \Phi(U) \subset S$ be a simple, piecewise
regular, closed curve with vertices $c(t_i)$ and exterior angles
$\beta_i$, $i=0,\dots,k$. Let
$\psi_i: [t_i,t_{i+1}] \rightarrow \mathbb R$ be the differentiable
function that measures at each $t\in[t_i,t_{i+1}]$ the oriented angle
between $\frac{\partial\Phi}{\partial u_1}$ and $c(t)$. Then,
$$
\sum_{i=0}^k
(\psi_i(t_{i+1})-\psi_i(t_i))
=
\pm2\pi-\sum_{i=0}^k\beta_i,
$$

where the sign depends on the orientation of $c$.
\end{lemma}

\begin{lemma}\label{bounded tau}
Let $(M,g)$ be a closed oriented surface of genus greater than one and
without focal points, and suppose that the magnetic flow on the energy
level $E=\frac12$ has no conjugate points. Then for every
$\widetilde\theta\in T_1\widetilde M$ there exist a constant
$C=C(\widetilde\theta)>0$ and an unbounded set $S\subset[0,\infty)$ such
that
\[
\bigl|\tau_{\widetilde\theta}(t)\bigr|\leq C
\qquad\text{for every }t\in S .
\]
\end{lemma}

\begin{proof}
Keep the notation above and let $t_\ell,\rho_\ell$ be as in
Lemma~\ref{ultima salida}. Fix $\ell>t_0$. We bound
$\tau_{\widetilde\theta}(t_\ell)$ by a constant independent of $\ell$;
since $t_\ell\geq\ell$, the set $S$ is then unbounded and the lemma follows.

\smallskip
\smallskip
\noindent\textbf{Step 1: a Busemann orbit avoiding
$\widetilde\sigma_{\widetilde\theta}$.}

Since $\widetilde\sigma_{\widetilde\theta}$ and $\gamma$ are at Hausdorff
distance less than $D$, and $u$ is the Busemann level of
$\widetilde\sigma_{\widetilde\theta}$, the portion of
$\widetilde\sigma_{\widetilde\theta}[0,\infty)$ lying in the strip
$b_\gamma^{-1}[-\ell,0]$ stays within $\widetilde g$-distance $D$ of the
compact arc $\gamma[-D,\ell+D]$. Let $s_0=s_0(-D,\ell+D,D)$ be given by
Lemma~\ref{faixa plana} and set
\[
r^*:=\max\bigl\{s_0,\rho_0,\rho_\ell\bigr\}+1 .
\]
Since $r^*\geq s_0$, the Busemann orbit
$c_2(t)=\Phi_{\widetilde\theta,\gamma}(t,r^*)$ satisfies
\[
c_2[-D,\ell+D]\cap V_D\bigl(\gamma[-D,\ell+D]\bigr)=\emptyset ,
\]
and $r^*>\max\{\rho_0,\rho_\ell\}$. Hence $c_2[0,\ell]$ is disjoint from
$\widetilde\sigma_{\widetilde\theta}[0,\infty)$. Indeed, $c_2[0,\ell]$ is
contained in the strip $b_\gamma^{-1}[-\ell,0]$, since $b_\gamma(c_2(t))=-t$ by
\eqref{busemann coord}. A point of $\widetilde\sigma_{\widetilde\theta}$
outside the strip therefore cannot lie on $c_2[0,\ell]$, and a point inside
the strip lies in the tube $V_D(\gamma[-D,\ell+D])$, which $c_2$ avoids.

\smallskip
\noindent\textbf{Step 2: the rectangle.}
Set
\[
P_1=\Phi_{\widetilde\theta,\gamma}(0,\rho_0)
=\widetilde\sigma_{\widetilde\theta}(t_0),
\qquad
P_2=\Phi_{\widetilde\theta,\gamma}(\ell,\rho_\ell)
=\widetilde\sigma_{\widetilde\theta}(t_\ell),
\]
\[
Q_1=\Phi_{\widetilde\theta,\gamma}(0,r^*),
\qquad
Q_2=\Phi_{\widetilde\theta,\gamma}(\ell,r^*),
\]
and consider the four sides
\[
\begin{aligned}
c_1(r)&=\Phi_{\widetilde\theta,\gamma}(0,r),
&& r\in[\rho_0,r^*],\\
c_2(t)&=\Phi_{\widetilde\theta,\gamma}(t,r^*),
&& t\in[0,\ell],\\
c_3(r)&=\Phi_{\widetilde\theta,\gamma}(\ell,r^*+\rho_\ell-r),
&& r\in[\rho_\ell,r^*],\\
c_4&=\bigl(\widetilde\sigma_{\widetilde\theta}|_{[t_0,t_\ell]}\bigr)^{-1},
\end{aligned}
\]
together with the closed, piecewise $C^1$ curve
\[
\Sigma_\ell:=c_1*c_2*c_3*c_4 ,
\]
which runs $P_1\to Q_1\to Q_2\to P_2\to P_1$ and is well defined because
$t_0<t_\ell$ by Lemma~\ref{ultima salida}.

\smallskip
\noindent\textbf{Step 3: $\Sigma_\ell$ is simple.}
The sides $c_1,c_2,c_3$ are coordinate arcs, so by the
injectivity of $\Phi_{\widetilde\theta,\gamma}$ they
meet one another exactly at $Q_1$ and $Q_2$. It remains to show that
$c_4$ meets them only at $P_1$ and $P_2$.

Suppose $\widetilde\sigma_{\widetilde\theta}(l)\in c_1$ with
$l\in[t_0,t_\ell]$. Then $u(l)=0$, so $l\in I_0$ and hence $v(l)\leq\rho_0$
by the maximality defining $\rho_0$; since every point of $c_1$ has
$r\geq\rho_0$, we get $v(l)=\rho_0$, and uniqueness in
Lemma~\ref{ultima salida} gives $l=t_0$. The same argument at the level
$\ell$ shows that any intersection with $c_3$ forces $l=t_\ell$.
Finally $c_4\cap c_2=\emptyset$ by Step 1. Hence $\Sigma_\ell$ is
simple.

\smallskip
\noindent\textbf{Step 4: Hopf's formula.}
Let $X_{\widetilde\theta}=\partial_t\Phi_{\widetilde\theta,\gamma}$ and
$Y_{\widetilde\theta}=\partial_r\Phi_{\widetilde\theta,\gamma}$ be the
coordinate fields, which form an orthogonal positively oriented frame, and
recall that $\tau_{\widetilde\theta}$ measures the oriented angle from
$X_{\widetilde\theta}$ to the velocity of the curve. Along $c_1$,
$c_2$ and $c_3$ the velocity is a positive multiple of $Y$, $X$
and $-Y$ respectively, so this angle is \emph{constant}, equal to $\pi/2$,
$0$ and $-\pi/2$; in particular the corresponding \emph{variations} vanish.
Along $c_4$ the angle is $\tau_{\widetilde\theta}+\pi$, traversed
backwards, so its variation equals
$\tau_{\widetilde\theta}(t_0)-\tau_{\widetilde\theta}(t_\ell)$.\\
Applying Lemma~\ref{Hopf's Lemma} to $\Sigma_\ell$ we obtain
\[
\tau_{\widetilde\theta}(t_0)-\tau_{\widetilde\theta}(t_\ell)
=\pm2\pi-\sum_{i=1}^{4}\beta_i ,
\]
where $\beta_1,\dots,\beta_4$ are the exterior angles of $\Sigma_\ell$ at
$Q_1,Q_2,P_2,P_1$ and the sign is determined by the orientation of
$\Sigma_\ell$. Since $|\beta_i|\leq\pi$, the right-hand side is bounded by
$6\pi$ independently of $\ell$, whence
\[
\bigl|\tau_{\widetilde\theta}(t_\ell)\bigr|
\leq\bigl|\tau_{\widetilde\theta}(t_0)\bigr|+6\pi=:C .
\]
\end{proof}

\begin{figure}[h!]
\centering
\begin{tikzpicture}[scale=1.2,>=stealth]

\fill[black!8] (-1.2,-0.85) rectangle (6.9,0.85);
\node[black!45,right,font=\footnotesize] at (6.55,-0.62) {$V_D(\gamma)$};

\draw[black!22] (-1.2,1.5) -- (6.9,1.5);
\draw[black!22] (-1.2,2.2) -- (6.9,2.2);

\draw[thick] (-1.2,0) -- (6.9,0) node[right] {$\gamma$};

\draw[black!55] (0,-1.05) -- (0,3.45);
\draw[black!55] (5,-1.05) -- (5,3.45);
\node[below,font=\footnotesize] at (0,-1.1) {$H^+_\gamma(0)$};
\node[below,font=\footnotesize] at (5,-1.1) {$H^+_\gamma(\ell)$};

\draw[very thick] (0,3) -- (5,3);
\node[above] at (2.5,3) {$c_2$};
\draw[very thick] (0,0.45) -- (0,3);
\node[right] at (0.12,1.75) {$c_1$};
\draw[very thick] (5,3) -- (5,0.52);
\node[right] at (5.12,1.75) {$c_3$};

\draw[thick,black!65] plot[smooth,tension=0.75] coordinates {
  (0.45,-0.70) (-0.30,-0.62) (-0.85,-0.48) (-0.55,-0.42) (0,-0.35)
  (0.55,-0.18) (0.20,0.10) (-0.30,0.30) (0,0.45)
  (1.1,0.52) (1.9,0.12) (2.7,0.40) (3.6,0.00) (4.4,0.28)
  (5,-0.15) (5.70,-0.05) (5.95,0.25) (5.40,0.45) (5,0.52)};

\draw[thick,black!65] plot[smooth,tension=0.8] coordinates {
  (5,0.52) (5.35,0.72) (5.95,0.55) (6.55,0.75) (6.95,0.60)};
\node[black!65] at (2.5,0.78) {$\widetilde\sigma_{\widetilde\theta}$};
\node[black!65,font=\footnotesize] at (2.5,-0.55) {$c_4$ (reversed)};

\fill (0,0.45) circle (1.8pt) node[above left] {$P_1$};
\fill (5,0.52) circle (1.8pt) node[above right] {$P_2$};
\fill (0,3)    circle (1.8pt) node[above left] {$Q_1$};
\fill (5,3)    circle (1.8pt) node[above right] {$Q_2$};

\draw[black!65,thick,fill=none] (0,-0.35) circle (2.2pt);
\draw[black!65,thick,fill=none] (5,-0.15) circle (2.2pt);

\draw[<->,black!50] (-0.75,0.45) -- (-0.75,3);
\node[left,font=\footnotesize] at (-0.75,1.72) {$r^*-\rho_0$};

\end{tikzpicture}
\caption{The curve $\Sigma_\ell$ in the horocyclic chart
$\Phi_{\widetilde\theta,\gamma}$. Vertical lines are horocycles, horizontal
lines are Busemann orbits. The open dots are the other intersections of
$\widetilde\sigma_{\widetilde\theta}[0,\infty)$ with $H^+_\gamma(0)$ and
$H^+_\gamma(\ell)$; the vertices $P_1,P_2$ are the ones of maximal $r$,
which is what makes $\Sigma_\ell$ simple.}
\label{fig:rectangulo}
\end{figure}
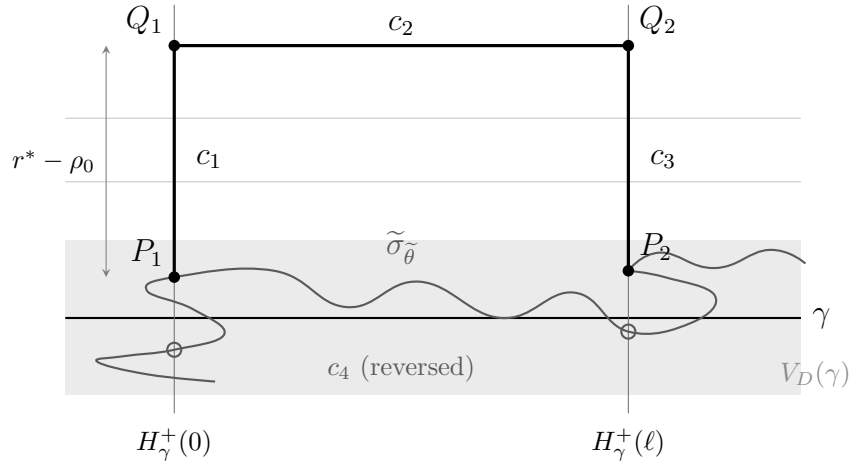


    

\begin{proof}[Proof of Theorem \ref{Teo A}]
Let $\theta\in T_1M$, let $\widetilde\theta$ be a lift, and set
\[
F(\theta)=(f\circ\pi)(\theta)-k(\theta)\sin\bigl(\tau(\theta)\bigr).
\]
By Corollary~\ref{k-em-M} and Lemma~\ref{curvatura=lorentz}, Liouville's
formula along $\widetilde\sigma_{\widetilde\theta}$ reads
\[
\frac1T\int_0^TF\bigl(\psi_s(\theta)\bigr)\,ds
=\frac{\tau_{\widetilde\theta}(T)-\tau_{\widetilde\theta}(0)}{T}.
\]
By Birkhoff's Ergodic Theorem the limit
$F^*(\theta)=\lim_{T\to\infty}\frac1T\int_0^TF(\psi_s(\theta))\,ds$ exists
$\mu_l$-a.e.\ and $\int_{T_1M}F\,d\mu_l=\int_{T_1M}F^*\,d\mu_l$. By
Lemma~\ref{bounded tau} there are $C>0$ and $T_n\to\infty$ with
$|\tau_{\widetilde\theta}(T_n)|\leq C$, so the right-hand side above tends
to $0$ along $T_n$; since the limit exists, $F^*=0$ $\mu_l$-a.e.\ and
therefore
\[
\int_{T_1M}(f\circ\pi)\,d\mu_l
=\int_{T_1M}k\,\sin(\tau)\,d\mu_l .
\]
By Corollary~\ref{horociclo curvatura no-negativa} we have $k\geq0$, whence
\[
\left|\int_{T_1M}(f\circ\pi)\,d\mu_l\right|
\leq\int_{T_1M}k\,|\sin(\tau)|\,d\mu_l
\leq\int_{T_1M}k\,d\mu_l .
\]
The result follows from $\int_{T_1M}(f\circ\pi)\,d\mu_l
=2\pi\int_Mf\,d\operatorname{vol}_g$.
\end{proof}


\begin{thebibliography}{99} 




\bibitem{BCS} Bao, D., Chern, S. S.,  Shen, Z. (2012). An introduction to Riemann-Finsler geometry (Vol. 200). Springer Science and Business Media.

\bibitem{BP2002} Burns, Keith, and Gabriel P. Paternain. "Anosov magnetic flows, critical values and topological entropy." Nonlinearity 15.2 (2002): 281-314.







\bibitem{Eberlein1973}Eberlein, Patrick. "When is a geodesic flow of Anosov type? I." Journal of Differential Geometry 8.3 (1973): 437-463.

\bibitem{EberleinONeill1973}Eberlein, Patrick, and Barrett O’Neill. "Visibility manifolds." Pacific Journal of Mathematics 46.1 (1973): 45-109.
\bibitem{Eschenburg1977} Eschenburg, Jost-Hinrich. "Horospheres and the stable part of the geodesic flow." (1977).


\bibitem{GomesRuggiero2013} Gomes, J. B.; Ruggiero, R. O.
\textit{On Finsler surfaces without conjugate points.}
Ergodic Theory Dynam. Systems \textbf{33} (2013), no.~2, 455--474.

\bibitem{Green1958} Green, L. W. \textit{A theorem of E. Hopf.} Michigan Math. J. \textbf{5} (1958), 31--34.


\bibitem{Heintze1977} Heintze, Ernst, and Hans-Christoph Im Hof. "Geometry of horospheres." Journal of Differential Geometry 12.4 (1977): 481-491.










\bibitem{Paternain1994}Paternain, Gabriel P. "On Anosov energy levels of Hamiltonians on twisted cotangent bundles." Boletim da Sociedade Brasileira de Matemática-Bulletin/Brazilian Mathematical Society 25.2 (1994): 207-211.
\bibitem{Pesin1977}Pesin, Ja B. "Geodesic flows on closed Riemannian manifolds without focal points." Mathematics of the USSR-Izvestiya 11.6 (1977): 1195-1228.

\bibitem{Paternain2006} Paternain, Gabriel. "Magnetic rigidity of horocycle flows." Pacific journal of mathematics 225.2 (2006): 301-323.







\bibitem{PS}Peyerimhoff, N., and K. F. Siburg. "The dynamics of magnetic flows for energies above Mané’s critical value." Israel Journal of Mathematics 135.1 (2003): 269-298.
\end{thebibliography}
\end{document}